\documentclass[oneside,english]{amsart}
\usepackage[T1]{fontenc}
\usepackage[latin9]{inputenc}
\usepackage{babel}
\usepackage{amsthm}
\usepackage{amstext}
\usepackage{amssymb}
\usepackage[unicode=true]
 {hyperref}
\usepackage{breakurl}

\makeatletter

\providecommand{\tabularnewline}{\\}

\numberwithin{equation}{section}
\numberwithin{figure}{section}
\theoremstyle{plain}
\newtheorem{thm}{\protect\theoremname}
  \theoremstyle{definition}
  \newtheorem{defn}[thm]{\protect\definitionname}
  \theoremstyle{plain}
  \newtheorem{lem}[thm]{\protect\lemmaname}
  \theoremstyle{definition}
  \newtheorem{example}[thm]{\protect\examplename}
  \theoremstyle{plain}
  \newtheorem{cor}[thm]{\protect\corollaryname}

\makeatother

  \providecommand{\corollaryname}{Corollary}
  \providecommand{\definitionname}{Definition}
  \providecommand{\examplename}{Example}
  \providecommand{\lemmaname}{Lemma}
\providecommand{\theoremname}{Theorem}

\begin{document}

\title{Eigenvalues of the Adin-Roichman Matrices}

\author{Gil Alon}
\begin{abstract}
We find the spectrum of Walsh-Hadamard type matrices defined by R.Adin
and Y.Roichman in their recent work on character formulas and descent
sets for the symmetric group.
\end{abstract}



\maketitle

\section{Introduction}

Adin and Roichman described in \cite{Adin and Roichman} a general
framework for various character formulas for representations of the
symmetric group. A key ingredient in their description is a family
of matrices- $\left(A_{n}\right)_{n\geq0}$ and $\left(B_{n}\right)_{n\geq0}$
which are defined, recursively, by $A_{0}=B_{0}=\left(1\right)$ and
for $n\geq1,$ 
\[
A_{n}=\left(\begin{array}{cc}
A_{n-1} & A_{n-1}\\
A_{n-1} & -B_{n-1}
\end{array}\right);\,\,\,\, B_{n}=\left(\begin{array}{cc}
A_{n-1} & A_{n-1}\\
0 & -B_{n-1}
\end{array}\right)
\]

The matrix $A_{n}$ was shown to connect between combinatorial objects
of various types, and character values of the symmetric group. It
has been shown in \cite{Adin and Roichman} that $A_{n}$ is invertible,
and hence the character formulas may be inverted, yielding formulas
for counting combinatorial objects with a given descent set using
character values. 

In this paper we find the eigenvalues of the matrices $A_{n}$ and
$B_{n}$, proving the following theorem:
\begin{thm}
\label{main}\begin{itemize} \item[$(i)$] The roots of the characteristic polynomial of $A_n$ are in $2:1$ correspondence with the compositions of $n$: each composition $\mu = (\mu_1, \ldots, \mu_t)$ of $n$ corresponds to a pair of eigenvalues $\pm \sqrt{\pi_\mu}$ of $A_n$, where \[ \pi_\mu := \prod_{i=1}^{t} (\mu_i+1). \] \item[$(ii)$] Similarly, the roots of the characteristic polynomial of $B_n$ are  in $2:1$ correspondence with the compositions of $n$: each composition  $\mu = (\mu_1, \ldots, \mu_t)$ of $n$ corresponds to a pair of eigenvalues $\pm \sqrt{\pi'_\mu}$ of $B_n$, where \[ \pi'_\mu := \prod_{i=1}^{t-1} (\mu_i+1). \] \end{itemize}
\end{thm}
These eigenvalues were conjectured by Adin and Roichman in the early
version of their paper \cite{Adin and Roichman}.

Our method of proof is as follows: We conjugate the matrices $A_{n}$
and $B_{n}$ by a combinatorially defined matrix $U_{n}$, and get
a lower anti-triangular matrix, i.e. a matrix with zeros above the
secondary diagonal. The area below the secondary diagonal in the conjugated
matrices is quite sparse, and we show that a permutation can be chosen,
so that after conjugating with the corresponding permutation matrix
we get matrices which are lower-triangular in blocks of size $2\times2$.
From this form, the eigenvalues can be easily obtained.

The rest of this paper is organized as follows: In section \ref{sec:Preliminaries}
we give some definitions, and recall the non-recursive definition
of $A_{n}$ and $B_{n}$ from \cite{Adin and Roichman}. In section
\ref{sec:Strategy-of-the}, we outline in more detail the strategy
of the proof. In section \ref{sec:Conjugation-by-Un}, we describe
the conjugation of $A_{n}$ and $B_{n}$ into lower anti-triangular
matrices. In section \ref{sec:Conjugation-by-perm}, we find the suitable
permutation, and in section \ref{sec:Eigenvalues} we use it to prove
the main theorem. In the final section we give an equivalent, non-recursive
description of the permutation involved in the proof, and also describe
it in terms of the Thue-Morse sequence.

\section{Preliminaries\label{sec:Preliminaries}}

Let us recall some of the definitions in \cite{Adin and Roichman}.
We use the notation $[n]=\{1,2,..,n\}$ and $[a,b]=\{i\in\mathbb{Z}|a\leq i\leq b\}$.
A nonempty set of the form $[a,b]$ is called an \emph{interval}.
Given intervals $I_{1}$ and $I_{2}$, $I_{1}$ is called a \emph{prefix}
of $I_{2}$ if $\min I_{1}=\min I_{2}$ and $I_{1}\subseteq I_{2}$.

Given a set $I\subseteq[n]$, the \emph{runs of $I$ }are the maximal
intervals contained in $I$. They are denoted, in ascending order,
by $I_{1},I_{2},...$ . For example, if $I=\{2,4,5\}$ then $I_{1}=\{2\}$
and $I_{2}=\{4,5\}$.
\begin{defn}
Given sets $I,J\subseteq[n]$, let us write $I\gg J$ if each run
of $I\cap J$ is a prefix of a run of $I$. (Note that this is not
an order relation).
\end{defn}
Let us now describe the non-recursive definition of $A_{n}$ and $B_{n}$.
It is convenient to index the rows and columns of the $2^{n}\times2^{n}$
matrices $A_{n}$ and $B_{n}$ by subsets of $[n]$. We order the
subsets of $[n]$ linearly by the lexicographical order, as described
in \cite{Adin and Roichman}. An equivalent definition of the lexicographical
order comes from the following function:
\begin{defn}
Let $r_{n}:P([n])\rightarrow[2^{n}]$ be given by the binary representation,
\[
r_{n}(A)=1+\sum_{i\in A}2^{i-1}.
\]

Note that $A\leq B$ with respect to the lexicographical order if
and only if $r_{n}(A)\leq r_{n}(B)$.
\begin{lem}
The matrices $A_{n}$ and $B_{n}$ are given by
\[
A_{n}(I,J)=\begin{cases}
(-1)^{|I\cap J|} & \text{if }I\gg J\\
0 & \text{otherwise}
\end{cases}
\]

and
\[
B_{n}(I,J)=\begin{cases}
(-1)^{|I\cap J|} & \text{if }I\gg J\text{ and }n\notin I\setminus J\\
0 & \text{otherwise}
\end{cases}
\]

\end{lem}
\end{defn}
The lemma is proved in \cite[lemma 3.8]{Adin and Roichman}.

\section{\label{sec:Strategy-of-the}Strategy of the proof}
\begin{defn}
Let $U_{n}$ be the matrix $U_{n}(I,J)=\begin{cases}
\begin{array}{c}
1\\
0
\end{array} & \begin{array}{c}
I\supseteq J\\
\textrm{otherwise}
\end{array}\end{cases}$
\end{defn}
Note that $U_{n}$ is the transpose of the matrix $Z_{n}$ defined
in \cite{Adin and Roichman}.
\begin{defn}
A matrix $A$ of size $m\times m$ is called \emph{lower anti-triangular
}if $A_{i,j}=0$ for all $i,j$ satisfying $i\leq n-j$.
\end{defn}
We will show that the matrices $U_{n}A_{n}U_{n}^{-1}$ and $U_{n}B_{n}U_{n}^{-1}$
are (when rows and columns are written in lexicographical order) anti-triangular.
For example,

$$U_3A_3U_3^{-1}= \begin{pmatrix}  0&0&0&0&0&0&0&1\cr  0&0&0&0&0&0&2&0\cr  0&0&0&0&0&2&0&0\cr  0&0&0&0&3&0&1&0\cr  0&0&0&2&0&0&0&0\cr  0&0&4&0&0&0&0&0\cr  0&3&0&0&0&1&0&0\cr  4&0&2&0&2&0&0&0\cr  \end{pmatrix} $$Furthermore,
we will see that $U_{n}A_{n}U_{n}^{-1}$ can be conjugated by a permutation
matrix, such that the resulting matrix is block-triangular with blocks
of size $2\times2$.

For example, for $n=3$ we may take the permutation 
\[
\sigma=\left(\begin{array}{cccccccc}
1 & 2 & 3 & 4 & 5 & 6 & 7 & 8\\
6 & 3 & 2 & 7 & 4 & 5 & 8 & 1
\end{array}\right).
\]
 If $P$ is the corresponding permutation matrix, then $$PU_{3}A_{3}U_{3}^{-1}P^{-1}=
\begin{pmatrix}  0&4&0&0&0&0&0&0\cr  2&0&0&0&0&0&0&0\cr  0&0&0&2&0&0&0&0\cr  1&0&3&0&0&0&0&0\cr  0&0&0&1&0&3&0&0\cr  0&0&0&0&2&0&0&0\cr  0&2&0&0&0&2&0&4\cr  0&0&0&0&0&0&1&0\cr \end{pmatrix}$$

It is easy to deduce the eigenvalues of $A_{n}$ from this form.

\section{\label{sec:Conjugation-by-Un}Conjugation by $U_{n}$}
\begin{defn}
Let $A_{n}'=U_{n}A_{n}(U_{n})^{-1}$ and $B_{n}'=U_{n}B_{n}(U_{n})^{-1}$.
\end{defn}
We will prove below that $A_{n}'$ and $B_{n}'$ are anti-triangular.
Some other properties of these matrices may be observed. For example,
for $n=9$ and $I=\{1,2,3,6,7,9\}$, we may note that $A_{9}'(I,J)\neq0$
only for 
\[
J=\{4,5,8\},\{2,4,5,8\},\{3,4,5,8\},\{4,5,7,8\},\{2,4,5,7,8\},\{3,4,5,7,8\}
\]
that is, only for sets of the form $\bar{I}\cup E$ where $E\subseteq I$
and $E$ does not contain a minimal element of a run of $I$, nor
does it contain two consecutive elements.
\begin{defn}
For a set $I\subseteq[n]$, let us denote $\pi(I)=\prod_{i}(n_{i}+1)$,
where $n_{i}$ is the size of the $i$th run $I_{i}$.\end{defn}
\begin{example}
$\pi(\{1,2,3,5\})=(3+1)(1+1)=8$.\end{example}
\begin{lem}
We have\label{lem:We-have}
\begin{enumerate}
\item $A_{n}'(I,J)=0$ if $I\cup J\neq[n]$. 
\item $A_{n}'(I,\bar{I})=\pi(I)$ (where $\bar{I}=[n]\setminus I$).
\item $A_{n}'(I,\bar{I}\cup E)=0$ if $E\subseteq I$ and $E$ contains
a minimal element of a run of $I$.
\item $A_{n}'(I,\bar{I}\cup E)=0$ if $E\subseteq I$ and $i,i+1\in E$
for some $i$.
\end{enumerate}
\end{lem}
\begin{proof}
By Möbius inversion (see \cite{Rota} and \cite[section 4]{Adin and Roichman}),
the inverse of $U_{n}$ is given by
\[
(U_{n})^{-1}(I,J)=\begin{cases}
(-1)^{|I\setminus J|} & \text{if }I\supseteq J\\
0 & \text{otherwise}
\end{cases}
\]

Hence, by definition of matrix multiplication,

\begin{eqnarray*}
A_{n}'(I,L) & = & \sum_{J,K\subseteq[n]}U_{n}(I,J)\cdot A_{n}(J,K)\cdot(U_{n}^{-1})(K,L)\\
 & = & \sum_{J,K:I\supseteq J,J\gg K,K\supseteq L}(-1)^{|J\cap K|+|K\setminus L|}
\end{eqnarray*}

We will use the last formula in the proof of each claim. Given $I,J,L$
and some $x\in[n]$, let us say that \emph{we may toggle $x$ in $K$}
if for each $K\subseteq[n]\setminus\{x\}$, the contributions of $(J,K)$
and $(J,K\cup\{x\})$ to the above sum cancel out. Similarly, given
$I,K$,$L$ and $x\in[n]$, we will say that \emph{we may toggle $x$
in $J$} if for each $J\subseteq[n]\setminus\{x\}$, the contributions
of $(J,K)$ and $(J\cup\{x\},K)$ to the above sum cancel out. 
\begin{enumerate}
\item Let us assume that $I\cup L\neq[n]$, and let $x\in[n]$ be such that
$x\notin I$ and $x\notin L$. For each $J$ in the sum $\sum_{J,K:I\supseteq J,J\gg K,K\supseteq L}(-1)^{|J\cap K|+|K\setminus L|}$,
we have $x\notin J$. We may toggle $x$ in $K$, since for $K\subseteq[n]\setminus\{x\},$
$K$ and $K\cup\{x\}$ have the same intersection with $J$. Hence,
the entire sum is $0$.
\item We have 
\[
A_{n}'(I,\bar{I})=\sum_{J,K:I\supseteq J,J\gg K,K\supseteq\bar{I}}(-1)^{|J\cap K|+|K\cap I|}.
\]
 For each $J\subsetneq I$, the sum over $J$ is $0$, since we may
take $x\in I\setminus J$ and toggle $x$ in $K$ (as in the previous
case, adding $x$ to $K$ does not change the intersection $K\cap J$).
Hence only $J=I$ contributes to the sum and we get 
\[
A_{n}'(I,\bar{I})=\sum_{K:I\gg K,K\supseteq\bar{I}}(-1)^{|I\cap K|+|K\cap I|}=\pi(I)
\]

\item Suppose that $x\in E$ is a minimal element of a run of $I$. For
each $J$ participating in the sum 
\[
A_{n}'(I,\bar{I}\cup E)=\sum_{J,K:I\supseteq J,J\gg K,K\supseteq\bar{I}\cup E}(-1)^{|J\cap K|+|K\cap I\cap\bar{E}|}
\]
we have $x-1\notin I\Rightarrow x-1\notin J$, whereas $x\in K$ (because
$x\in E$). Hence, we may toggle $x$ in $J$ (adding $x$ to $J$
may only extend one run in $J$ one place to the left, and removing
$x$ may only shrink one run by one place from the left, hence $J\gg K\Leftrightarrow J\cup\{x\}\gg K$),
and the whole sum is $0$.
\item If $i,i+1\in E$ then in the sum for $A_{n}'(I,\bar{I}\cup E)$, for
each $K$ in the sum we have $i,i+1\in I\cap K$ and we may toggle
$i+1$ in $J$.
\end{enumerate}
\end{proof}
\begin{defn}
Let $I,J\subseteq[n]$.
\begin{enumerate}
\item We write $J\preceq I$ if $J\subseteq I$ and $J$ does not contain
any minimal element of a run of $I$, nor does it contain two consecutive
elements.
\item We write $I\curvearrowright J$ if \textbf{$J=\bar{I}\cup E$} for
some $E\preceq I$.
\end{enumerate}
\end{defn}
\begin{cor}
\label{cor:1}$\,$
\begin{enumerate}
\item $A_{n}'$ is lower anti-triangular.
\item If $A_{n}'(I,J)\neq0$ then $I\curvearrowright J$ .
\end{enumerate}
\end{cor}
This follows immediately from lemma \ref{lem:We-have}.

Similar results hold for $B_{n}'$:
\begin{defn}
For a set $I\subseteq[n]$, let us denote $\pi_{n}'(I)=\prod_{i}(n_{i}+1)$,
where $n_{i}$ is the size of the $i$th run $I_{i}$, and the product
excludes the run containing $n$, if it exists. \end{defn}
\begin{example}
$\pi_{8}'(\{1,2,4,5,7,8\})=(2+1)(2+1)=9$ and $\pi_{8}'(\{1,2,4,6,7)=(2+1)(1+1)(2+1)=18$.\end{example}
\begin{lem}
We have\label{lem:We-have-1}
\begin{enumerate}
\item $B_{n}'(I,J)=0$ if $I\cup J\neq[n]$. 
\item $B_{n}'(I,\bar{I})=\pi_{n}'(I)$.
\item $B_{n}'(I,\bar{I}\cup E)=0$ if $E\subseteq I$ and $E$ contains
a minimal element of a run of $I$.
\item $B_{n}'(I,\bar{I}\cup E)=0$ if $E\subseteq I$ and $i,i+1\in E$
for some $i$.
\end{enumerate}
\end{lem}
\begin{proof}
The proof goes along the lines of the proof of lemma \ref{lem:We-have}.
We have
\[
B_{n}'(I,L)=\sum_{J,K:I\supseteq J,J\gg K,K\supseteq L,n\notin J\setminus K}(-1)^{|J\cap K|+|K\setminus L|}
\]
We repeat the arguments in the above proof:
\begin{enumerate}
\item Let $x\in[n]$ be such that $x\notin I$ and $x\notin L$. Since for
each $J$ in the sum, $x\notin J$, adding or removing $x$ from $K$
does not change $J\setminus K$. Hence, we may still toggle $x$ in
$K$.
\item We have 
\[
B_{n}'(I,\bar{I})=\sum_{J,K:I\supseteq J,J\gg K,K\supseteq\bar{I},n\notin J\setminus K}(-1)^{|J\cap K|+|K\cap I|}.
\]
If $J\subsetneq I$, we may still take $x\in I\setminus J$ and toggle
$x$ in $K$ (adding or removing $x$ from $K$ will not change $J\setminus K$).
Hence we may take $J=I$ in the sum: 
\[
B_{n}'(I,\bar{I})=\sum_{K:I\gg K,K\supseteq\bar{I},n\notin I\setminus K}(-1)^{|I\cap K|+|K\cap I|}=\pi_{n}'(I).
\]

\item We have
\[
B_{n}'(I,\bar{I}\cup E)=\sum_{J,K:I\supseteq J,J\gg K,K\supseteq\bar{I}\cup E,n\notin J\setminus K}(-1)^{|J\cap K|+|K\cap I\cap\bar{E}|}.
\]
If $x\in E$ is a minimal element of a run of $I$, then for each
$K$ participating in the sum for $B_{n}(I,\bar{I}\cup E)$ we have
$x\in K$, and toggling $x$ in $J$ does not change $J\setminus K$. 
\item Again, if $i,i+1\in E$ then $i+1\in K$ and we may still toggle $i+1$
in $J$.
\end{enumerate}
\end{proof}
\begin{cor}
$\,$\label{cor:2}
\begin{enumerate}
\item $B_{n}'$ is lower anti-triangular.
\item If $B_{n}'(I,J)\neq0$ then $I\curvearrowright J$ .
\end{enumerate}
\end{cor}

\section{Conjugation by a permutation matrix\label{sec:Conjugation-by-perm}}

We have shown that $A_{n}$ is conjugate to a matrix $A_{n}'$ which
satisfies some nice properties: It is anti-triangular, its anti-diagonal
elements are given by $A_{n}(I,\bar{I})=\pi(I)$, and it is sparse:
$A_{n}'(I,J)\neq0$ implies $I\curvearrowright J$. We will use all
these properties to find a suitable permutation matrix for further
conjugating $A_{n}'$ into a $2\times2$ block-triangular matrix.
The same permutation will also conjugate $B_{n}'$ into a block matrix
of the same type.
\begin{lem}
\label{lem2}There exists a one-to-one function $\sigma_{n}:[2^{n}]\rightarrow P([n])$
such that:
\begin{enumerate}
\item For all $1\leq i\leq2^{n-1}$, $\sigma_{n}(2i)=\overline{\sigma_{n}(2i-1)}$
\item For all $1\leq i\leq2^{n-1},1\in\sigma_{n}(2i-1)$
\item If $\sigma_{n}(i)\curvearrowright\sigma(j)$ then $\sigma_{n}(i)=\overline{\sigma_{n}(j)}$
or $j\leq i$.
\end{enumerate}
\end{lem}
Simply put, the lemma states that we can list the subsets of $[n]$
in pairs of complementing sets, such that when a set $I$ is listed,
all the sets $J$ such that $I\curvearrowright J$, except possibly
$\overline{I}$, have already been listed.
\begin{example}
For $n=3$, we may take the following function:

\begin{tabular}{|c|c|c|}
\hline 
$i$ & $\sigma_{3}(2i-1)$ & $\sigma_{3}(2i)$\tabularnewline
\hline 
\hline 
$1$ & $\{1,3\}$ & $\{2\}$\tabularnewline
\hline 
$2$ & $\{1\}$ & $\{2,3\}$\tabularnewline
\hline 
$3$ & $\{1,2\}$ & $\{3\}$\tabularnewline
\hline 
$4$ & $\{1,2,3\}$ & $\emptyset$\tabularnewline
\hline 
\end{tabular}

In fact, this is the only possible function satisfying the requirements
of lemma \ref{lem2}, but for larger values of $n$ this function
is not always unique. Note, however that for any $n$ we must have
$\sigma_{n}(1)=\{1,3,5,...\}$ (since by definition of the relation
$\curvearrowright$, for any set $I\neq\{1,3,5,...\}$ such that $1\in I$,
there exists a set $J\notin\{I,\bar{I}\}$ such that $I\curvearrowright J$
or $\bar{I}\curvearrowright J$, hence $\sigma_{n}(1)$ cannot be
equal to $I$).\end{example}
\begin{proof}
We shall construct a function $\sigma_{n}$ satifying the above conditions
explicitely.

The construction is recursive: For $n=1$ we define $\sigma_{1}(1)=\{1\},\sigma_{1}(2)=\emptyset$.

Let us assume that $\sigma_{1},...,\sigma_{n-1}$ have been defined. 

First we define the value of $\sigma_{n}(j)$ for $j\leq2^{n-1}$
by: 
\begin{eqnarray*}
\sigma_{n}(2i-1) & = & \{1\}\cup(\sigma_{n-1}(2i)+1)\\
\sigma_{n}(2i) & = & \sigma_{n-1}(2i-1)+1\\
(1\leq i\leq2^{n-2})
\end{eqnarray*}

We have used the notation $I+x=\{i+x|i\in I\}$ for a set $I$ and
a number $x$.

Note that all the pairs of sets in this half-list have $1$ in one
set and $2$ in the other.

We define the next $2^{n-2}$ values by

\begin{eqnarray*}
\sigma_{n}(2^{n-1}+2i-1) & = & \{1,2\}\cup(\sigma_{n-2}(2i)+2)\\
\sigma_{n}(2^{n-1}+2i) & = & \sigma_{n-2}(2i-1)+2\\
(1\leq i\leq2^{n-3})
\end{eqnarray*}

and in general,

\begin{eqnarray*}
\sigma_{n}(2^{n}-2^{n-k}+2i-1) & = & [k+1]\cup(\sigma_{n-k-1}(2i)+k+1)\\
\sigma_{n}(2^{n}-2^{n-k}+2i) & = & \sigma_{n-k-1}(2i-1)+k+1
\end{eqnarray*}

for all $0\leq k\leq n-2,1\leq i\leq2^{n-k-2}$. 

Finally, we define
\begin{eqnarray*}
\sigma_{n}(2^{n}-1) & = & [n]\\
\sigma_{n}(2^{n}) & = & \emptyset
\end{eqnarray*}

Let us call the sets defined at the $k$-th stage (i.e. the sets at
places $2^{n}-2^{n-k}+1,..,2^{n}-2^{n-k}+2^{n-k-1}$) \emph{the sets
of the $k$-th chunk}. Since all the functions $\sigma_{i}$ are one-to-one,
the $k$-th chunk consists of all the sets that contain $[1,k+1]$
but don't contain $\{k+2\}$, and the complements of these sets (which
are exactly the sets whose minimum is $k+2$). Hence, $\sigma_{n}$
is one-to-one.

Let us prove that the conditions are satisfied: the first two, $\sigma_{n}(2i)=\overline{\sigma_{n}(2i-1)}$
and $1\in\sigma_{n}(2i-1)$, are immediate by the construction. For
the third one, we look again at 
\[
I:=\sigma_{n}(2^{n}-2^{n-k}+2i-1)=[k+1]\cup(\sigma_{n-k-1}(2i)+k+1)
\]

If $I\curvearrowright J$ and $J\neq\bar{I}$, then we may write \textbf{$J=\bar{I}\cup E$}
for some $\emptyset\neq E\preceq I$. If $[k+1]\cap E=\emptyset$,
then $\min J=k+2$, hence $J$ also belongs to the $k$th chunk, and
by the induction hypothesis, since $(I\setminus[k+1])-(k+1)\curvearrowright J-(k+1)$,
$J$ is equal to $\sigma_{n}(2^{n}-2^{n-k}+2j)$ for some $1\leq j<i$.
If $[k+1]\cap E\neq\emptyset$, then let $l=\min E$. We have $1<l\leq k+1$
($1$ cannot be an element of $E$ since $E\preceq I$) and $\bar{I}\cup E$
belongs to the $(l-2)$nd chunk, hence (since $l-2<k$) appears before
$I$ in the list.

Next, we consider

\[
J:=\sigma_{n}(2^{n}-2^{n-k}+2i)=\sigma_{n-k-1}(2i-1)+k+1
\]
Note that $\min J=k+2$. Given $\emptyset\neq E\preceq J$, \textbf{$\bar{J}\cup E$}
contains $[1,k+1]$ and does not contain $k+2$, so belongs to the
$k$-th chunk. We have $k+2\notin E$, hence $1\notin E-(k+1)$. Also,
$k+2\in J$ and by the induction hypothesis $K:=([n-k-1]\setminus(J-(k+1)))\cup(E-(k+1))$
appears before $J-(k+1)$ in the list $\sigma_{n-k-1}$, namely $K=\sigma_{n-k-1}(2j)$
for some $j<i$. Hence, \textbf{$\bar{J}\cup E=[k+1]\cup(K+k+1)$
}appears before $J=(J-(k+1))+k+1$ in $\sigma_{n}$. 
\end{proof}

\section{\label{sec:Eigenvalues}Eigenvalues of $A_{n}$ and $B_{n}$}

Let us take $\sigma_{n}$ as in lemma \ref{lem2} and view $\sigma_{n}$
as a permutation on $[2^{n}]$ (using, as usual, the lexicographical
order on $P([n])$). Let $P_{n}$ be the permutation matrix corresponding
to $\sigma_{n}$, i.e. $P_{n}(i,j)=\delta_{\sigma_{n}(i),j}$, and
let $A_{n}''=P_{n}A_{n}'P_{n}^{-1}$. Then $A_{n}''$ is conjugate
to $A_{n}$.

We have $A_{n}''(i,j)=P_{n}A_{n}'P_{n}^{-1}(i,j)=A_{n}'(\sigma_{n}(i),\sigma_{n}(j))$,
hence (by Corollary \ref{cor:1} and Lemma \ref{lem2}) $A_{n}''(i,j)=0$
if $j>i$ and $\{i,j\}$ is not of the form $\{2t+1,2t+2\}$. Hence,
$A_{n}''$ is lower triangular in $2\times2$ blocks. 

The blocks on the main diagonal of $A_{n}''$ are in correspondence
with subsets of $I\subseteq[n]$ satisfying $1\in I$. To such a set
corresponds the block 
\[
\left(\begin{array}{cc}
0 & \pi(I)\\
\pi(\bar{I}) & 0
\end{array}\right)
\]

The characteristic polynomial of this block is $t^{2}-\pi(I)\pi(\bar{I)}$.
Hence the characteristic polynomial of $A_{n}$ is

\[
\det(tI-A_{n})=\det(tI-A_{n}'')=\prod_{1\in I\subseteq[n]}\left(t^{2}-\pi(I)\pi(\bar{I})\right)
\]

Similarly, let us define $B_{n}''=P_{n}B_{n}'P_{n}^{-1}$, and again
$B_{n}''$ is lower triangular in $2\times2$ blocks, the blocks on
the main diagonal are in one-to-one correspondence with subsets $1\in I\subseteq[n]$,
and the block corresponding to such $I$ is

\[
\left(\begin{array}{cc}
0 & \pi_{n}'(I)\\
\pi_{n}'(\bar{I}) & 0
\end{array}\right)
\]

Thus,

\[
\det(tI-B_{n})=\det(tI-B_{n}'')=\prod_{1\in I\subseteq[n]}\left(t^{2}-\pi_{n}'(I)\pi_{n}'(\bar{I})\right)
\]

Let us note that there is a one-to-one correspondence between sets
$1\in I\subseteq[n]$ and compositions of $n$. 
\begin{defn}
Given a set $I$ satisfying $1\in I\subseteq[n]$, let $n_{1},n_{2},...$
be the sizes of the runs $I_{1},I_{2},..$ of $I$, and let $m_{1},m_{2},...$
be the sizes of the runs $\bar{I}_{1},\bar{I}_{2},...$ of $\bar{I}$.
Let $\mu_{n}(I)$ be the composition $(n_{1},m_{1},n_{2},m_{2},...)$
of $n$. 
\end{defn}
For example, $\mu_{8}(\{1,4,5\})=(1,2,2,3)$.

The correspondence $\mu_{n}$ satisfies $\pi_{\mu_{n}(I)}=\pi(I)\pi(\bar{I})$
and $\pi'_{\mu_{n}(I)}=\pi_{n}'(I)\pi_{n}'(\bar{I})$ (recall the
definitions of $\pi_{\mu}$ and $\pi'_{\mu}$ for a composition $\mu$
in the statement of Theorem \ref{main}).

We conclude:
\begin{thm}
\label{thm:We-have}We have 
\begin{itemize}
\item $\det(tI-A_{n})=\prod_{\mu}\left(t^{2}-\pi_{\mu}\right)$
\item $\det\left(tI-B_{n}\right)=\prod_{\mu}\left(t^{2}-\pi'_{\mu}\right)$
\end{itemize}
The products extend over all compositions $\mu$ of $n$.
\end{thm}
This proves Theorem \ref{main}.

\section{\label{sec:A-closer-look}A closer look at $\sigma_{n}$}

The function $\sigma_{n}:[2^{n}]\rightarrow P([n])$ has been defined
recursively in the proof of Lemma \ref{lem2}. We will now give a
non-recursive definition. For that purpose, it is more convenient
to look at the permutation $\sigma_{n}r_{n}:P([n])\rightarrow P([n])$
(recall the definition of $r_{n}$ in section \ref{sec:Preliminaries}).
\begin{thm}
\label{thm:The-permutation}The permutation $\sigma_{n}r_{n}$ is
given by
\[
t\in\sigma_{n}r_{n}(I)\Leftrightarrow|\left(\{1\}\cup[n-t+2,n]\right)\setminus I|\equiv1\mod2
\]
\end{thm}
\begin{proof}
Let us prove by induction on $n.$ For $n=1$, we have $\sigma_{1}r_{1}(\emptyset)=\sigma_{1}(1)=\{1\}$
and $\sigma_{1}r_{1}(\{1\})=\sigma_{1}(2)=\emptyset$, and accordingly,
$|\{1\}\setminus\emptyset|=1$ and $|\{1\}\setminus\{1\}|=0.$

Suppose that the claim is true for $1,..,n-1$. Recall the recursive
definition
\[
\sigma_{n}(2^{n}-2^{n-k}+2i-1)=[k+1]\cup(\sigma_{n-k-1}(2i)+k+1)
\]

Let $2i-1=r_{n-k-1}(J)$, for some $J\subseteq[n-k-1]$ (note that
$1\notin J$).

Since $2^{n}-2^{n-k}=2^{n-k}+2^{n-k+1}+..+2^{n-1}$, we have

\[
2^{n}-2^{n-k}+2i-1=r_{n}(J\cup[n-k+1,n]).
\]

Let 
\[
I:=J\cup[n-k+1,n].
\]

Note that $n-k\notin J$ and $n-k\notin I$. We have

\begin{eqnarray*}
\sigma_{n}r_{n}(I) & = & [k+1]\cup(\sigma_{n-k-1}(2i)+k+1)\\
 & = & [k+1]\cup(\sigma_{n-k-1}r_{n-k-1}(\{1\}\cup J)+k+1)
\end{eqnarray*}

For all $t\in[n]$, 
\begin{itemize}
\item if $t\leq k+1$ then (since $1\notin I)$,
\[
|\left(\{1\}\cup[n-t+2,n]\right)\setminus I|=1
\]
 and $t\in\sigma_{n}r_{n}(I)$.
\item If $t=k+2$, then (since $n-k\notin I$), 
\[
|\left(\{1\}\cup[n-t+2,n]\right)\setminus I|=2
\]
 and (since $1\notin\sigma_{n-k-1}(2i)$), $t\notin\sigma_{n}r_{n}(I)$.
\item If $t>k+2$ then 
\[
|\left(\{1\}\cup[n-t+2,n]\right)\setminus I|=|\left(\{1\}\cup[n-t+2,n-k]\right)\setminus J)|.
\]
We have $2i=r_{n-k-1}(\{1\}\cup J)$ and by the induction hypothesis,
\begin{eqnarray*}
 & t\in\sigma_{n}r_{n}(I)\\
\Leftrightarrow & t-(k+1)\in\sigma_{n-k-1}r_{n-k-1}(\{1\}\cup J)\\
\Leftrightarrow & |\left(\{1\}\cup[(n-k-1-(t-(k+1))+2,n-k-1]\right)\setminus(\{1\}\cup J)|\equiv1\mod2\\
\Leftrightarrow & |[n-t+2,n-k-1]\setminus J|\equiv1\mod2\\
\Leftrightarrow & |\left(\{1\}\cup[n-t+2,n]\right)\setminus I|\equiv1\mod2
\end{eqnarray*}
as desired (in the last stage we used the facts that $1\notin I,\, n-k\notin I$
and $[n-k+1,n]\subseteq I$). 
\end{itemize}

Also, $\sigma_{n}r_{n}([2,n])=\sigma_{n}(2^{n}-1)=[n]$, and $|\left(\{1\}\cup[n-t+2,n]\right)\setminus[2,n]|\equiv1\mod2$
for all $t$, as desired.

Thus, we have verified the claim for all $I\subseteq[n]$ such that
$1\notin I$. The case $1\in I$ follows easily from the previous
case and the property $\sigma_{n}(2i)=\overline{\sigma_{n}(2i-1)}$.

\end{proof}
Another description of $\sigma_{n}$ has to do with the Thue-Morse
sequence. Let us recall the definition of the sequence: It is a binary
sequence $\left(t_{n}\right)_{n\geq0}$, obtained as the limit of
the finite words $w_{n}$, where $w_{0}=0$ and $w_{n+1}=w_{n}\overline{w_{n}}$
for $n\geq0$. For example, $w_{1}=01,w_{2}=0110,w_{3}=01101001$,
hence the first terms of the sequence are $0,1,1,0,1,0,0,1,..$ .

The term $t_{n}$ of the Thue-Morse sequence is equal, modulo $2$,
to the sum of the digits in the binary expansion of $n$ (See \cite{Allouch and Shallit}).

Let us encode the function $\sigma_{n}$ with $n$ binary words of
length $2^{n}$:
\begin{defn}
For $1\leq j\leq n$, let $W_{n,j}$ be the binary word $w_{1,j}w_{2,j}...w_{2^{n},j}$
where $w_{i,j}=\begin{cases}
1 & j\in\sigma_{n}(i)\\
0 & \textrm{otherwise}
\end{cases}$.

Then $W_{n,j}$ is given in terms of the Thue-Morse sequence:\end{defn}
\begin{cor}
$W_{n,j}=\begin{cases}
(t_{0}\overline{t_{0}})^{2^{n-j}}(t_{1}\overline{t_{1}})^{2^{n-j}}...(t_{2^{j-1}-1}\overline{t_{2^{j-1}-1}})^{2^{n-j}} & 2|j\\
(\overline{t_{0}}t_{0})^{2^{n-j}}(\overline{t_{1}}t_{1})^{2^{n-j}}...(\overline{t_{2^{j-1}-1}}t_{2^{j-1}-1})^{2^{n-j}} & 2\nmid j
\end{cases}$\end{cor}
\begin{proof}
Let $I\subseteq[n]$ and $i=r_{n}(I)$. Since $|\{1\}\cup[n-j+2,n]|=j$,
we have by Theorem \ref{thm:The-permutation}, 
\begin{eqnarray*}
j\in\sigma_{n}(i) & \Leftrightarrow\\
|\left(\{1\}\cup[n-j+2,n]\right)\setminus I|\equiv1\mod2 & \Leftrightarrow\\
|\left(\{1\}\cup[n-j+2,n]\right)\cap I|\equiv j+1\mod2 & \Leftrightarrow\\
(i-1)+t_{\left\lfloor \frac{i-1}{2^{n+1-j}}\right\rfloor }\equiv j+1\mod2 & \Leftrightarrow\\
t_{\left\lfloor \frac{i-1}{2^{n+1-j}}\right\rfloor }\equiv i+j\mod2
\end{eqnarray*}

The result follows immediately.\end{proof}

\end{document}